\documentclass[12pt]{amsart}
\usepackage{fullpage}
\usepackage{tabmac}
\usepackage{amsthm,amssymb}
\usepackage{hyperref}
\usepackage{epsfig}
\usepackage{pstricks}
\usepackage{pst-plot}
\usepackage{pdflscape}
\usepackage{color}

\newtheorem{theorem}{Theorem}[section]
\newtheorem{thm}[theorem]{Theorem}
\newtheorem{lemma}[theorem]{Lemma}
\newtheorem{lem}[theorem]{Lemma}

\newtheorem{prop}[theorem]{Proposition}

\newtheorem{cor}[theorem]{Corollary}
\newtheorem{conjecture}[theorem]{Conjecture}

\theoremstyle{remark}
\newtheorem{remark}{Remark}[section]

\theoremstyle{definition}
\newtheorem{example}[theorem]{Example}

\def\image{{\rm image}}
\def\e{{\mathfrak e}}
\def\eb{\mathfrak{eb}}
\def\ec{\mathfrak{ec}}

\def\i{{\mathbf i}}
\def\j{{\mathbf j}}
\def\Z{\mathbb{Z}}

\def\P{{\mathcal P}}

\def\ep{\varepsilon}

\def\R{{\mathbb R}}
\def\el{\mathfrak{el}}
\def\ssl{\mathfrak{sl}}
\def\sp{\mathfrak{sp}}
\def\gl{\mathfrak{gl}}

\def\i{{\mathbf i}}

\begin{document}

\author{Thomas Lam}
\email{tfylam@umich.edu}
\address{Department of Mathematics,
University of Michigan, 530 Church St., Ann Arbor, MI 48109 USA}
\author{Pavlo Pylyavskyy}
\email{ppylyavs@umn.edu}
\address{Department of Mathematics,
University of Minnesota, 206 Church St. SE, Minneapolis, MN 55455 USA}
\thanks{T.L. was partially supported by NSF grant DMS-0901111, and by a Sloan Fellowship.}
\thanks{P.P. was partially
supported by NSF grant DMS-0757165.}

\title{Electrical networks and Lie theory}

\begin{abstract}
We introduce a new class of ``electrical'' Lie groups.  These Lie groups, or more precisely their nonnegative parts, act on the space of planar electrical networks
via combinatorial operations previously studied by Curtis-Ingerman-Morrow.  The corresponding electrical Lie algebras 
are obtained by deforming the Serre relations of a semisimple Lie algebra in a way suggested by the star-triangle transformation of electrical networks.  Rather surprisingly, we 
show that the type $A$ electrical Lie group is isomorphic to the symplectic group.  The electrically nonnegative part $(EL_{2n})_{\geq 0}$ of the electrical Lie group is an analogue of the totally nonnegative subsemigroup $(U_{n})_{\geq 0}$ of the unipotent subgroup of $SL_{n}$.  We establish decomposition and parametrization results for 
$(EL_{2n})_{\geq 0}$, paralleling Lusztig's work in total nonnegativity, and work of Curtis-Ingerman-Morrow and de Verdi\`{e}re-Gitler-Vertigan for networks.   Finally, we suggest a generalization of electrical Lie algebras to all Dynkin types. 
\end{abstract}
\maketitle

\section{Introduction}
In this paper we consider the simplest of electrical networks, namely those that consist of only resistors.  The electrical properties of such a network $N$ are completely described 
by the {\it response matrix} $L(N)$, which computes the current that flows through the network when certain voltages are fixed at the boundary vertices of $N$.  
The study of the response matrices of {\it planar} electrical networks has led to a robust theory, see Curtis-Ingerman-Morrow \cite{CIM}, or de Verdi\`{e}re-Gitler-Vertigan \cite{dVGV}. 
 In 1899, Kennelly \cite{Ken} described a local transformation (see Figure \ref{fig:elec11}) of a network, called the {\it star-triangle} or {\it $Y-\Delta$} transformation, which preserves the response matrix of a network.  This transformation is one of the many places where a Yang-Baxter style transformation occurs in mathematics or physics.  

\medskip

Curtis-Ingerman-Morrow \cite{CIM} studied the operations of {\it adjoining a boundary spike} and {\it adjoining a boundary edge} to (planar) electrical networks.  
Our point of departure is to consider these operations as one-parameter subgroups of a Lie group action.  The star-triangle transformation then leads to an ``electrical Serre relation'' in the corresponding Lie algebra, which turns out to be a deformation of the Chevalley-Serre relation for $\ssl_{n}$:
$$
\text{Serre relation:}\; [e,[e,e']] = 0  \qquad \text{electrical Serre relation:} \;[e,[e,e']] = -2e
$$
The corresponding one-parameter subgroups satisfy a Yang-Baxter style relation which is a deformation of Lusztig's relation in total positivity:
\begin{align*}
\text{Lusztig's relation:} &\;u_i(a) u_j(b) u_i(c) = u_j({bc}/({a+c})) u_i(a+c) u_j({ab}/({a+c})) \\
\text{electrical relation:} &\; u_i(a) u_j(b) u_i(c) = u_j({bc}/({a+c+abc})) u_i(a+c+abc) u_j({ab}/({a+c+abc})).
\end{align*}
We study in detail the Lie algebra $\el_{2n}$ with $2n$ generators satisfying the electrical Serre relation.  An (electrically) nonnegative part $(EL_{2n})_{\geq 0}$ of the 
corresponding Lie group $EL_{2n}$ acts on planar electrical networks with $n+1$ boundary vertices, and one obtains a dense subset of all response matrices of planar networks 
in this way.  This nonnegative part is a rather precise analogue of the totally nonnegative subsemigroup $(U_{2n+1})_{\geq 0}$ of the unipotent subgroup of $SL_{2n+1}$, studied in Lie-theoretic terms by Lusztig \cite{Lus}.  We show (Proposition \ref{P:decomp}) that $(EL_{2n})_{\geq 0}$ has a cell decomposition labeled by permutations $w \in S_{2n+1}$, precisely paralleling one of Lusztig's results for $(U_{2n})_{\geq 0}$ and reminiscent of the Bruhat decomposition.  This can be considered an algebraic analogue of parametrization results in the theory of electrical networks \cite{CIM,dVGV}.  Surprisingly, $EL_{2n}$ itself is isomorphic to the symplectic Lie group $Sp_{2n}(\R)$ (Theorem \ref{thm:elsp}). This semisimplicity does not in general hold for electrical Lie groups: $EL_{2n+1}$ is not semisimple.  We also describe (Theorem \ref{thm:stab}) the stabilizer lie algebra of the infinitesimal action of $\el_{2n}$ on electrical networks.  We caution that the nonnegative part $(EL_{2n})_{\geq 0}$ is distinct from Lusztig's totally nonnegative part $Sp_{2n}(\R)_{\geq 0}$ of the symplectic group.

While we focus on the planar case in this paper, we shall connect the results of this paper to the inverse problem for electrical networks on cylinders in a future paper.
There we borrow ideas from representation theory such as that of crystals and $R$-matrices.

\medskip
One obtains the types $B$ and $G$ electrical Serre relations by the standard technique of ``folding'' the type $A$ electrical Serre relation.  
These lead to a new species of electrical Lie algebras $\e_D$ defined for any Dynkin diagram $D$.  Besides the above results for type $A$, there appear to be other interesting 
relations between these $\e_D$ and simple Lie algebras.  For example, $\eb_2 := \e_{B_2}$ is isomorphic to $\gl_2$.  We conjecture (Conjecture \ref{con:1}) that for a finite 
type diagram $D$, the dimension $\dim(\e_D)$ is always equal to the dimension of the maximal nilpotent subalgebra of the semisimple Lie algebra ${\mathfrak g}_D$ with Dynkin 
diagram $D$, and furthermore (Conjecture \ref{con:2}) that $\e_D$ is finite dimensional if and only if ${\mathfrak g}_D$ is finite-dimensional.  We give the electrical analogue 
of Lusztig's relation for type $B$ (see Berenstein-Zelevinsky \cite{BZ}) where we observe similar positivity to the totally nonnegative case. 

The most interesting case beyond finite Dynkin types is affine type $A$. The corresponding electrical Lie (semi)-group action on planar electrical networks is perhaps even more natural than that of $EL_{2n}$, since one can obtain {\it {all}} (rather than just a dense subset of, see Corollary \ref{cor:dense}) response matrices of planar networks in this way (\cite{dVGV, CIM}), and furthermore the circular symmetry of the planar networks is preserved.  We however do not attempt to address this case in the current paper. 

\section{Electrical networks}
For more background on electrical networks, we refer the reader to \cite{CIM, dVGV}.
\subsection{Response matrix}
For our purposes, an electrical network is a finite weighted undirected graph $N$, where the vertex set is divided into the {\it boundary} vertices and the {\it interior} vertices.  The weight $w(e)$ of an edge is to be thought of as the conductance of the corresponding resistor, and is generally taken to be a positive real number.  A $0$-weighted edge would same as having no edge, and an edge with infinite weight would be the same as identifying the endpoint vertices.  

We define the {\it Kirchhoff matrix} $K(N)$ to be the square matrix with rows and columns labeled by the vertices of $N$ as follows:
$$
K_{ij} = \begin{cases} -\sum_{e \; \text{joins $i$ and $j$}} w(e) &\mbox{for $i \neq j$} \\
\sum_{e \; \text{incident to $i$}} w(e) &\mbox{for $i = j$.}
\end{cases}
$$

We define the {\it response matrix} $L(N)$ to be the square matrix with rows and columns labeled by the boundary vertices of $N$, given by the Schur-complement
$$
L(N) = K/K_I
$$
where $K_I$ denotes the submatrix of $K$ indexed by the interior vertices.  The response matrix encodes all the electrical properties of $N$ that can be measured from the boundary vertices.  For example, the Kirchhoff matrix and response matrix of the ``Y"-graph in Figure \ref{fig:elec11} (with the surrounding edges removed and the central vertex made interior), are given by
$$
K(N) = \begin{bmatrix}a&0&0&-a \\ 0&b&0&-b \\ 0&0&c&-c \\ a &b&c&a+b+c\end{bmatrix} \qquad \text{and} \qquad
L(N) = \begin{bmatrix}
 \frac{a (b+c)}{a+b+c} & -\frac{a b}{a+b+c} & -\frac{a c}{a+b+c} \\
 -\frac{a b}{a+b+c} & \frac{b (a+c)}{a+b+c} & -\frac{b c}{a+b+c} \\
 -\frac{a c}{a+b+c} & -\frac{b c}{a+b+c} & \frac{(a+b) c}{a+b+c} \\
\end{bmatrix}.
$$

\subsection{Planar electrical networks}
We shall usually consider electrical networks $N$ embedded into a disk, so that the boundary vertices, numbered $1,2,\ldots,n+1$, lie on the boundary of the disk.  
\begin{figure}[h!]
    \begin{center}
    \input{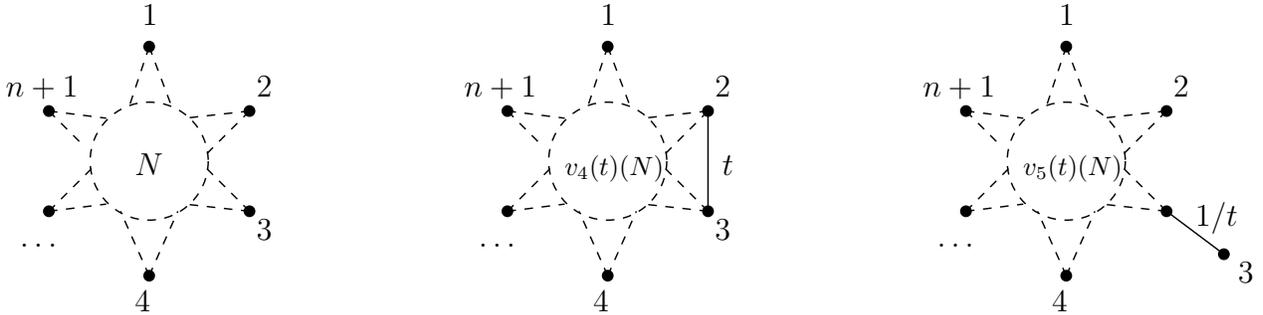}
    \end{center}
    \caption{The operations $v_i(t)$ acting on a network $N$.}
    \label{fig:eLie3}
\end{figure}
Given an odd integer $2k-1$, for $k = 1,2,\ldots,n+1$, and a nonnegative real number $t$, we define $v_{2k-1}(t)(N)$ to be the electrical network obtained from $N$ by adding a new edge from a new vertex $v$ to $k$, with weight $1/t$, followed by treating $k$ as an interior vertex, and the new vertex $v$ as a boundary vertex (now named $k$).

Given an even integer $2k$, for $k = 1,2,\ldots,n+1$, and a nonnegative real number $t$, we define $v_{2k}(t)(N)$ to be the electrical network obtained from $N$ by adding a new edge 
from $k$ to $k+1$ (indices taken modulo $n+1$), with weight $t$.

The two operations are shown in Figure \ref{fig:eLie3}. In \cite{CIM}, these operations are called {\it adjoining a boundary spike}, and {\it adjoining a boundary edge} respectively.  
Our notation suggests, and we shall establish, that there is some symmetry between these two types of operations.  

In \cite[Section 8]{CIM}, it is shown that $L(v_i(a) \cdot N)$ depends only on $L(N)$, giving an operation $v_i(t)$ on response matrices.  Denote by $x_{ij}$  the entries of the response matrix, where 
$1 \leq i , j \leq n+1$. In particular, we have $x_{ij} = x_{ji}$. Then if $\delta_{ij}$ denotes the Kronecker delta, we have 
\begin{align} \label{eq:cim}
v_{2k-1}(t): x_{ij} \mapsto x_{ij} -\frac{t x_{ik} x_{kj}}{tx_{kk}+1} \;\;\; \text{and} \;\;\;  v_{2k}(t): x_{ij} \mapsto x_{ij} + (\delta_{ik}-\delta_{(i+1)k}) (\delta_{jk}-\delta_{(j+1)k}) t.
\end{align}
We caution that the parameter $\xi$ in \cite{CIM} is the inverse of our $t$ in the odd case.

\subsection{Electrically-equivalent transformations of networks} \label{ss:trans}
\label{sec:trans}

Series-parallel transformations are shown in Figure \ref{fig:eLie1}. The following proposition is well-known and can be found for example in \cite{dVGV}.

\begin{figure}[h!]
    \begin{center}
    \input{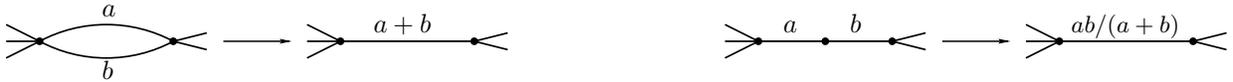}
    \end{center}
    \caption{Series-parallel transformations.}
    \label{fig:eLie1}
\end{figure}

\begin{prop}\label{P:SP}
Series-parallel transformations, removing loops, and removing interior degree 1 vertices, do not change the response matrix of a network.
\end{prop}


\begin{figure}[h!]
    \begin{center}
    \input{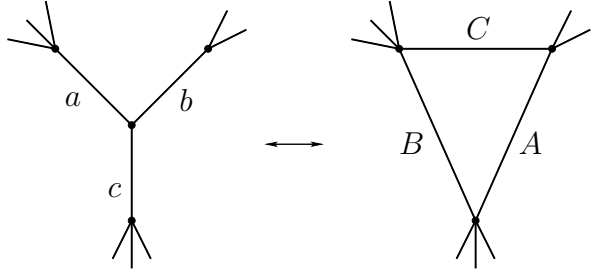}
    \end{center}
    \caption{The $Y-\Delta$, or star-triangle, transformation.}
    \label{fig:elec11}
\end{figure}

The following theorem is attributed to Kennelly \cite{Ken}.

\begin{prop}[$Y-\Delta$ transformation] \label{P:YDelta}
 Assume that parameters $a$,$b$,$c$ and $A$,$B$,$C$ are related by 
$$A = \frac{bc}{a+b+c}, \;\; B = \frac{ac}{a+b+c}, \;\; C= \frac{ab}{a+b+c},$$
or equivalently by 
$$a = \frac{AB+AC+BC}{A}, \;\; b= \frac{AB+AC+BC}{B}, \;\; c = \frac{AB+AC+BC}{C}.$$
Then switching a local part of an electrical network between the two options shown in Figure \ref{fig:elec11} does not change the response matrix of the whole network.
\end{prop}

\section{The electrical Lie algebra $\el_{2n}$}
\label{sec:el}
Let $\el_{2n}$ be the Lie algebra generated over $\R$ by $e_i$, $i=1, \ldots, 2n$ subject to relations 
\begin{itemize}
 \item $[e_i, e_j] = 0$ for $|i-j|>1$;
 \item $[e_i, [e_i, e_j]] = -2 e_i$, $|i-j|=1$.
\end{itemize}

\begin{theorem} \label{thm:elsp}
The Lie algebra $\el_{2n}$ is isomorphic to the real symplectic Lie algebra $\sp_{2n}$.
\end{theorem}
\begin{proof}
We identify the symplectic algebra $\sp_{2n}$ with the space of $2n \times 2n$ matrices which in block form
$$
\left(\begin{array}{cc} A & B \\ C & D \end{array} \right)
$$
satisfy $A = -D^T$, $B = B^T$ and $C = C^T$ (see \cite[Lecture 16]{FH}).  Note that the subLie algebra consisting of the matrices where $B = C = 0$ is naturally isomorphic to $\gl_n$.

Let $\ep_i \in \R^n$ denote the standard basis column vector with a $1$ in the $i$-th position.  Let $a_1 = \ep_1, a_2 = \ep_1 + \ep_2,a_3 = \ep_2 + \ep_3, \ldots,a_n = \ep_{n-1} + \ep_n$, and let $b_1 = \ep_1, b_2 = \ep_2,\ldots, b_n = \ep_n$.  For $1 \leq i \leq n$, define elements of $\sp_{2n}$ by the formulae
$$
e_{2i-1} = \left(\begin{array}{cc} 0 & a_i \cdot a_i^T \\ 0 & 0 \end{array} \right) \qquad e_{2i} = \left(\begin{array}{cc} 0 & 0 \\ b_i \cdot b_i^T & 0 \end{array} \right).
$$

We claim that this gives a symplectic representation $\phi$ of $\el_{2n}$ which is an isomorphism of $\el_{2n}$ with $\sp_{2n}$.  We first check that the relations of $\el_{2n}$ are satisfied.  It is clear from the block matrix form that $[e_{2i-1},e_{2j-1}]=0=[e_{2i},e_{2j}]$ for any $i,j$.  Now, 
$$
[e_{2i-1},e_{2j}] = \left(\begin{array}{cc}(a_i \cdot a_i^T) (b_j \cdot b_j^T) & 0\\ 0 & -(b_j \cdot b_j^T)(a_i \cdot a_i^T)   \end{array} \right).
$$
But by construction, $b_j^T \cdot a_i= 0 = a_i^T \cdot b_j$ unless $j = i$, or $j = i-1$.  Thus $[e_k,e_l] = 0$ unless $|k-l|= 1$.  Finally, 
$$
[e_{2i-1},[e_{2i-1},e_{2i}]] = \left(\begin{array}{cc}0 &-2(a_i \cdot a_i^T)(b_i \cdot b_i^T)(a_i \cdot a_i^T)\\ 0 & 0  \end{array} \right) = \left(\begin{array}{cc}0 &-2 a_i \cdot a_i^T\\ 0 & 0  \end{array} \right) = -2 e_{2i-1}
$$
using the fact that $a_i^T \cdot b_i = 1 = b_i^T \cdot a_i$.  Similarly, one obtains $[e_{k},[e_{k},e_{k \pm1}]]= -2e_k$.  Thus we have a symplectic representation $\phi: \el_{2n} \to \sp_{2n}$.  

Now we show that this map is surjective.  First we verify that the $
\gl_n \subset \phi(\el_{2n})$.  The non-zero commutators $[e_i,e_j]$ gives matrices of the form $\left(\begin{array}{cc} A & 0 \\ 0&-A^T \end{array}\right)$ where $A$ is a scalar multiple of one of the matrices $E_{1,1}$, or $E_{i,i}+E_{i+1,i}$ or $E_{i+1,i}+E_{i+1,i+1}$.  Here $E_{i,j}$ denotes the $n \times n$ matrix with a single non-zero entry equal to one in the $(i,j)$-th position.  All the matrices of the above form occur.  It is easy to see that $\phi(\el_{2n})$ must then contain the matrices where $A= E_{i,i}$, $E_{i,i+1}$, or $E_{i+1,i}$ for each $i$.  But these matrices generate $\gl_n$ as a Lie algebra. However, it is known \cite[Proposition 8.4(d)]{Hum} that for a semisimple Lie algebra $L$, one has $[L_{\alpha},L_{\beta}] = L_{\alpha+\beta}$ when $\alpha,\beta,\alpha+\beta$ are all roots, and $L_\alpha$ denotes a root subspace.  It follows easily from the explicit description of the root system of $\sp_{2n}$ and the definition of $e_i$ that every root subspace of $\sp_{2n}$ is contained in $\phi(\el_{2n})$, completing the proof.


To see that the map $\phi:\el_{2n} \to \sp_{2n}$ is injective, we note:
\begin{lem}\label{lem:dimA}
The dimension of $\el_{2n}$ is $n(2n+1)$.
\end{lem}

\noindent {\it Proof.}  According to Lemma \ref{lem:dim} the dimension of $\el_{2n}$ is at most $n(2n+1)$. On the other hand, we just saw that the map $\el_{2n} \to \sp_{2n}$ is surjective. The statement follows. 
\end{proof}

\section{The electrical Lie group $EL_{2n}$}
Let $EL_{2n}$ be a split real Lie group with Lie algebra $\el_{2n}$.  For concreteness, we shall choose $EL_{2n}$ to be the real symplectic group, but we will use the notation $EL_{2n}$ to remind the reader that the generators we consider are associated to the presentation of $\el_{2n}$, rather than the usual presentation of $\sp_{2n}$.  Let $u_i(a) = \exp(a e_i)$ denote the the one-parameter subgroups corresponding to the generators of $\el_{2n}$.  

\begin{theorem} \label{thm:rels}
 The elements $u_i(a)$ satisfy the relations
\begin{enumerate}
 \item $u_i(a) u_i(b) = u_i(a+b)$,
 \item $u_i(a) u_j(b) = u_j(b) u_i(a)$, if $|i-j|>1$;
 \item $u_i(a) u_j(b) u_i(c) = u_j({bc}/({a+c+abc})) u_i(a+c+abc) u_j({ab}/({a+c+abc}))$, if $|i-j|=1$.
\end{enumerate}
\end{theorem}

\begin{proof} The first two relations are clear. For the third, observe that for each $1 \leq i \leq 2i-1$, the elements $e_i$, $e_{i+1}$, and $[e_i,e_{i+1}]$ are the usual Chevalley generators of a Lie subalgebra isomorphic to $\ssl_2$.  Thus we can verify the relation inside $SL_2(\R)$:
$$\left(\begin{matrix} 1 & a\\ 0 &  1\\ \end{matrix} \right) \left(\begin{matrix} 1 & 0\\ b &  1\\ \end{matrix} \right) \left(\begin{matrix} 1 & c\\ 0 &  1\\ \end{matrix} \right) = 
\left(\begin{matrix} 1+ab & a+c+abc\\ b &  1+bc\\ \end{matrix} \right) =$$ 
$$\left(\begin{matrix} 1 & 0\\ {bc}/({a+c+abc}) &  1\\ \end{matrix} \right) \left(\begin{matrix} 1 & a+c+abc\\ 0 &  1\\ \end{matrix} \right) \left(\begin{matrix} 1 & 0\\ {ab}/({a+c+abc}) &  1\\ \end{matrix} \right).$$
\end{proof}

\begin{remark}
There is a one-parameter deformation which connects the relation Theorem \ref{thm:rels}(3) with Lusztig's relation \cite{Lus} in total positivity:
\begin{equation}\label{E:lus}
u_i(a) u_j(b) u_i(c) = u_j({bc}/({a+c})) u_i(a+c) u_j({ab}/({a+c})).
\end{equation}
Consider the associative algebra $U_{\tau}({\el}_{2n})$ where the generators $\ep_1,\ep_2,\ldots,\ep_{2n}$ satisfy the following deformed Serre relations:
\begin{itemize}
 \item $\ep_i \ep_j = \ep_j \ep_i$ if $|i-j|>1$;
 \item $\ep_i \ep_j \ep_i = \tau \ep_i + (\ep_i^2 \ep_j + \ep_j \ep_i^2)/2$, if $|i-j|=1$.
\end{itemize}
This is a one-parameter family of algebras which at $\tau=0$ reduces to $U({\mathfrak n^+})$, where ${\mathfrak sl}_{2n+1} = {\mathfrak n^-} \oplus {\mathfrak h} \oplus {\mathfrak n^+}$ 
is the Cartan 
decomposition, while at $\tau=1$ it gives the universal enveloping algebra $U({\mathfrak el}_{2n})$
of the electrical Lie algebra.  For $U_{\tau}({\el}_{2n})$, the ``braid move'' for the elements $\exp(a \ep_i)$
then takes the form 
$$u_i(a) u_j(b) u_i(c) = u_j({bc}/({a+c+\tau abc})) u_i(a+c+\tau abc) u_j({ab}/({a+c+\tau abc})) \;\;\; \text{if $|i-j|=1$}.$$
At  $\tau = 0$ this reduces to \eqref{E:lus}, and at $\tau=1$ it is the relation Theorem \ref{thm:rels}(3).
\end{remark}

\subsection{Nonnegative part of $EL_{2n}$}
Define the {\it {nonnegative}} part $(EL_{2n})_{\geq 0}$ of $EL_{2n}$, or equivalently the {\it {electrically nonnegative}} part of $Sp_{2n}$, to be the subsemigroup generated by the $u_i(a)$ with nonnegative parameters $a$.

For a reduced word $\i = {i_1} \dotsc {i_\ell}$ of $w \in S_{2n+1}$ denote by $E(w) \subset (EL_{2n})_{\geq 0}$ the image of the map 
$$
\psi_\i: (a_1, \ldots, a_\ell) \in \mathbb R_{>0}^\ell \mapsto u_{i_1}(a_1) \dotsc u_{i_\ell}(a_\ell).
$$
It follows from the relations in Theorem \ref{thm:rels} that the set $E(w)$ depends only on $w$, and not on the chosen reduced word.
The following proposition is similar to \cite[Proposition 2.7]{Lus} which gives a decomposition $U_{\geq 0} = \sqcup_w U^w_{\geq 0}$ of the totally nonnegative part of the unipotent group.

\begin{prop}\label{P:decomp}
 The sets $E(w)$ are disjoint and cover the whole $(EL_{2n})_{\geq 0}$. Each of the maps $\psi_\i:\mathbb R_{>0}^\ell \to E(w)$ is a homeomorphism.
\end{prop}

\begin{proof}
Using Theorem \ref{thm:rels}, we can rewrite any product of generators $u_i(a)$ by performing braid moves similar to those in the symmetric group $S_{2n+1}$.  Any product of the generators can be transformed into a product that corresponds to a reduced word in $S_{2n+1}$, and thus it belongs to one of the $E(w)$. 

If the map $\psi_\i$ is not injective for some reduced word $\i$, then we can find two reduced products
$$u_{i_1}(a_1) \dotsc u_{i_\ell}(a_\ell) = u_{i_1}(b_1) \dotsc u_{i_\ell}(b_\ell)$$ for two $\ell$-tuples of positive numbers 
such that $a_1 \neq b_1$. Without loss of generality we can assume 
$a_1 > b_1$, and thus $$u_{i_1}(a_1-b_1) u_{i_2}(a_2) \dotsc u_{i_\ell}(a_\ell) = u_{i_2}(b_2) \dotsc u_{i_\ell}(b_\ell).$$ 
This shows that two different $E(w)$-s have non-empty intersection.  Thus it suffices to show that the latter is impossible. 

Furthermore, it suffices to prove that the top cell corresponding to the longest element $w_0$ does not intersect any other cell. Indeed, if any two cells intersect, by adding some extra factors to both 
we can lift it to the top cell intersecting one of the other cells.  Assume we have a product of the form
\begin{equation}\label{E:ured}
u = [u_1(a_1)] [u_2(a_2) u_1(a_3)] \dotsc [u_k(a_{\ell-k+1}) \dotsc u_1(a_{\ell})]
\end{equation} where $\ell = {k \choose 2}$.  Let $\Phi(u) \in Sp_{2n}$ be the image of $u$ in $Sp_{2n}$ under the natural map $EL_{2n} \to Sp_{2n}$ induced by the map $\phi:\el_{2n} \to \sp_{2n}$ of Theorem \ref{thm:elsp}.  We argue that the positive parameters $a_{\ell-k+1}, \ldots, a_\ell$ can be recovered uniquely 
from just looking at the $n+k/2$-th row of $\Phi(u)$ for $k$ even, or the $(k+1)/2$-th row for $k$ odd.  Furthermore, the same calculation will tell us if exactly one of them is equal to zero.

Indeed, assume $k$ is odd. Then the $n+(k+1)/2$-th entry in the $(k+1)/2$-th row is just equal to $a_{\ell-k+1}$. Next, once we know $a_{\ell-k+1}$,
 we can use $(k-1)/2$-th entry of the same row to recover $a_{\ell-k+2}$, after which we can use $n+(k-1)/2$ entry to recover $a_{\ell-k+3}$, and so on. For each step we solve a linear equation 
where the parameter we divide by is strictly positive as long as all previous parameters $a_i$ recovered are positive. For example, let $n=2$ and take the product 
$$u_1(a_1) u_2(a_2) u_1(a_3) u_3(a_4) u_2(a_5) u_1(a_6) .$$ 
Then the second row of this product is $(a_4 a_5, 1, a_4 + a_4 a_5 a_6, a_4)$. The last entry tells us $a_4$, then from 
the first entry we solve for $a_5$, then from the third entry we solve for $a_6$. The case of even $k$ is similar, the order in which we have to read the entries of the  $n+k/2$-th row
in this case is $k/2$-th, $n+k/2$-th, $k/2-1$-th, $n+k/2-1$-th, etc. For example, in the product
$$u_1(a_1) u_2(a_2) u_1(a_3) u_3(a_4) u_2(a_5) u_1(a_6) u_4(a_7) u_3(a_8) u_2(a_9) u_1(a_{10}),$$ the last row is $(a_7 a_8 a_9, a_7, a_7 a_8 + a_7 a_8 a_9 a_{10}, 1+a_7 a_8)$.
By looking at second, last, first and third entries we can solve for $a_7, a_8, a_9$ and $a_{10}$ one after the other. 

Now we are ready to complete the argument.  Assume that $E(w_0)$ intersects some other $E(v)$.  For any reduced word of $w_0$, one can find a subword which is a reduced word for $v$, so that we have
\begin{align*}
 u &= [u_1(a_1)] [u_2(a_2) u_1(a_3)] \dotsc [u_k(a_{\ell-k+1}) \dotsc u_1(a_{\ell})] \\ &= [u_1(b_1)] [u_2(b_2) u_1(b_3)] \dotsc [u_k(b_{\ell-k+1}) \dotsc u_1(b_{\ell})],
\end{align*}
where the $a_i$ are all positive, but the $b_i$ are nonnegative, with at least one zero.  The above algorithm of recovering the $a_i$'s will arrive at a contradiction.  Thus the $E(w_0)$ is disjoint from all other $E(v)$.

It remains to show that $\phi_\i^{-1}$ is continuous for any reduced word.  If $\i$ is the reduced word of $w_0$ used in \eqref{E:ured}, this follows from the algorithm above: the $a_i$ depend continuously on the matrix entries of $\Phi(u)$.  But any two reduced words are connected by braid and commutation moves, so it follows from the (continuously invertible) formulae in Theorem \ref{thm:rels} that $\phi_\i^{-1}$ is continuous for any reduced word of $w_0$.  But for any other $v \in S_{2n+1}$, a reduced word $\j$ for $v$ can be found as an initial subword of some reduced word $\i$ for $w_0$.  The map $\phi_\j^{-1}$ can then be expressed as a composition of: right multiplication by a fixed element of $EL_{2n}$, the map $\phi_\i^{-1}$, and the projection from $\R^{\ell(w_0)}_{>0}$ to $\R^{\ell(v)}_{>0}$, all of which are continuous.
\end{proof}

\subsection{Action on electrical networks}
Let $\P(n+1) \subset \R^{(n+1)^2}$ denote the set of response matrices of planar electrical networks with $n+1$ boundary vertices.  In \cite{CIM}, it is shown that $\P(n+1)$ is exactly the 
set of symmetric, circular totally-nonnegative $(n+1) \times (n+1)$ matrices with row sums equal to 0.  A matrix $M$ is circular totally-nonnegative if the signed minors $(-1)^k A_{(p_1,p_2,\ldots,p_k),(q_1,q_2,\ldots,q_k)}$ are all nonnegative, whenever $(p_1,p_2,\ldots,p_k,q_k,q_{k-1},\ldots,q_1)$ is in circular order.  In \cite[Th\'{e}or\`{e}me 4]{dVGV}, it is shown that $\P(n+1)$ can be identified with the set of planar 
electrical networks with $n+1$ boundary vertices modulo the local transformations of Section \ref{ss:trans}.  Let $N_0$ denote the empty network (with $n+1$ boundary vertices) and let $L_0 = L(N_0)$ denote the zero matrix.


\begin{theorem}\label{T:action}
The nonnegative part $(EL_{2n})_{\geq 0}$ of the electrical group acts on $\P(n+1)$ via:
$$
u_i(a) \cdot L(N) = L(v_i(a)(N)).
$$
\end{theorem}
\begin{proof}
For a single generator $u_i(a)$, the stated action is well-defined because it can be described explicitly on the level of response matrices.  The formulae for $L( v_i(a)(N))$ in terms of $L(N)$ is given by the equation \eqref{eq:cim}.

We first show that the relations of Theorem \ref{thm:rels} hold for this action.  Relation (1) follows from the series-parallel relation for networks.  Relation (2) is immediate: the corresponding networks are identical without any transformations.  Relation (3) follows from  the $Y-\Delta$ transformation (see Example \ref{ex:YDelta}).

But now suppose we have two different expressions for $u \in (EL_{2n})_{\geq 0}$ in terms of generators.  Then using Relations (1)-(3) of Theorem \ref{thm:rels}, we may assume that both expressions are products corresponding to a reduced word.  By Proposition \ref{P:decomp}, the two products come from reduced words for the same $w \in S_{2n+1}$.  It follows that they are related by relations (1)-(3).
\end{proof}

\begin{example}\label{ex:YDelta}
 The products $u_3(a) u_4(b) u_3(c)$ and $u_4(\frac{bc}{a+c+abc}) u_3(a+c+abc) u_4(\frac{ab}{a+c+abc})$ act on a network in exactly the same way, as shown in Figure \ref{fig:elec10}. 
\begin{figure}[h!]
    \begin{center}
    \input{eLie4.pstex_t}
    \end{center}
    \caption{}
    \label{fig:elec10}
\end{figure}
\end{example}
%

A permutation $w = w(1) w(2) \cdots w(2n+1) \in S_{2n+1}$ is {\it efficient} if 
\begin{enumerate}
\item
$w(1) < w(3) < \cdots < w(2n+1)$ and
$w(2) < w(4) < \cdots < w(2n)$
\item
$w(1) < w(2)$, $w(3) < w(4)$, $\ldots$, $w(2n-1) < w(2n)$.
\end{enumerate}
Recall the {\it left weak order} of permutations is given by $w \preceq v$ if and only if there is a $u$ so that $uw = v$ and $\ell(v) = \ell(u) + \ell(w)$.  It is a standard fact that $w \preceq v$ if and only if whenever $w(i) > w(j)$ and $i < j$ then $v(i) > v(j)$.  Thus the set of efficient permutations has a maximum in left weak order, namely $w = 1 (n+2) 2 (n+3) \cdots (n) (2n+1) (n+1)$ with length $n(n-1)/2$.

\begin{thm}\label{T:efficient}
Let $w \in S_{2n+1}$.  The map $\Theta_w: E(w) \to \P(n+1)$ given by $\Theta_w(u) = u \cdot L_0$ is injective if and only if $w$ is efficient.  We have $\image(\Theta_w) \cap \image(\Theta_v) = \emptyset$ for $w \neq v$ both efficient.  If $w$ is not efficient, there is a unique efficient $v$ such that $\image(\Theta_w) = \image(\Theta_v)$.
\end{thm}
\begin{proof}
Let $w^* \in S_{2n+1}$ be the efficient permutation of maximal length.  Then a possible reduced word for $w^*$ is $(n+1) \cdots \left(46 \cdots (2n-2)\right)\left(35 \cdots (2n-1)\right) \left(246 \cdots (2n)\right)$.  The graph obtained by taking the corresponding $u_i(a)$ and acting on the empty network is exactly the ``standard graph'' of \cite[Section 7]{CIM} or the graph $C_N$ of \cite{dVGV}.  In particular, $\Theta_{w^*}$ is injective by \cite[Theorem 2]{CIM} or \cite[Th\'{e}or\`{e}me 3]{dVGV}.
Suppose $w$ is an arbitrary efficient permutation.  
Then since $w^* = uw$ for some $u \in S_{2n+1}$, if $\Theta_w$ is not injective then $\Theta_{w^*}$ is not injective as well, so we conclude that $\Theta_w$ is injective.

For a pair $(i,j)$ with $1 \leq i < j \leq n+1$, let us say that a network $N$ is $(i,j)$-{\it {connected}} if we can find a disjoint set of paths $p_1, p_2,...,p_{\lfloor (j-i+1)/2\rfloor}$ 
so that for each $k$, the path
$p_k$ connects boundary vertex $i+k-1$ to boundary vertex $j-k+1$ without passing through any other boundary vertex.  This is a special case of the connections of circular pairs $(P,Q)$
 of \cite{CIM}.  Let $N_w$ be a graph constructed from some reduced word of an efficient $w$.  We observe that $N_w$ is $(i,j)$-connected if and only if $w(2i) > w(2j-1)$.
\begin{figure}[h!]
    \begin{center}
    \input{eLie5.pstex_t}
    \end{center}
    \caption{}
    \label{fig:eLie5}
\end{figure}
For example, the first network in Figure \ref{fig:eLie5} corresponds to $w = s_5 s_3 s_6 s_4 s_2 = (1,3,5,2,7,4,6)$. We see that it is $(2,4)$-connected, which agrees with 
$w(4)=6>5=w(7)$. If the dashed edge is not there, we have $w = s_3 s_6 s_4 s_2 = (1,3,5,2,4,7,6)$ and $w(4)=5<6=w(7)$ in agreement with the network not being $(2,4)$-connected. Similarly, 
the second network in Figure \ref{fig:eLie5} corresponds to $w = s_4 s_5 s_3 s_6 s_4 s_2 = (1,3,5,7,2,4,6)$ with the dashed edge and to $w = s_5 s_3 s_6 s_4 s_2 = (1,3,5,2,7,4,6)$ without. In the first case
it is $(1,4)$-connected, in the second it is not, in agreement with relative order of $w(2)$ and $w(7)$.

Note that the inversions $w(2i) > w(2j-1)$ are exactly the possible inversions of an efficient permutation.  Since $w$ is determined by its inversions, it follows that the set of $(i,j)$-connections 
of $N_w$ determines $w$, and that $\image(\Theta_w) \cap \image(\Theta_v) = \emptyset$ for $w \neq v$ both efficient.

Suppose $w$ is not efficient.  Then it is easy to see that $w$ has a reduced expression $s_{i_1}s_{i_2} \cdots s_{i_\ell}$ where either (1) $i_\ell$ is odd, or (2) $i_\ell$ is even 
and $i_\ell = i_{\ell-1} \pm 1$.  But $u_i(a) \cdot L_0 = L_0$ for odd $i$ since all the boundary vertices are still disconnected in $u_i(a) \cdot N_0$, and for even $i$ we have $u_{i\pm1}(a) u_{i}(b) \cdot L_0 = u_i(1/a + b) L_0$, using the series-transformation (Proposition \ref{P:SP}).  It follows that $\Theta_w$ is not injective.  Furthermore, $\image(\Theta_w) = \image(\Theta_v)$, where $v$ is obtained from $w$ by recursively (1) removing $i_\ell$ from a reduced word of $w$ if $i_\ell$ is odd, or (2) changing the last two letters $(i_{\ell}\pm 1) i_\ell$ to $i_\ell$ when $i_\ell$ is even.  An efficient $v$ obtained in this way must be unique, since $\image(\Theta_v) \cap \image(\Theta_{v'}) = \emptyset$ for efficient $v \neq v'$.
\end{proof}

\begin{cor}
The number of efficient $w \in S_{2n+1}$ is equal to the Catalan number $C_{n+1} = \frac{1}{n+2}{2n+2 \choose n+1}$.
\end{cor}
\begin{proof}
It is clear that $(2i,2j+1)$ is an inversion of $w$ only if $(2i,2j-1)$ and $(2i+2,2j+1)$ are also inversions, and this characterizes inversion sets of efficient $w$.  
Thus the set of efficient $w \in S_{2n+1}$ is in bijection with the lower order ideals of the positive root poset of $\ssl_{n+1}$ which is well known to be enumerated by the 
Catalan number \cite{FR}.
\end{proof}

\begin{cor}\label{cor:dense}
$(EL_{2n})_{\geq 0} \cdot L_0$ is dense in $\P(n+1)$.
\end{cor}
\begin{proof}
Follows from Theorem \ref{T:efficient} and \cite[Th\'{e}or\`{e}me 5]{dVGV}.
\end{proof}




\subsection{Stabilizer}

\begin{lem}
The stabilizer subsemigroup of $(EL_{2n})_{\geq 0}$ acting on the zero matrix $L_0$ is the subsemigroup generated by $u_{2i+1}(a)$ for all $a \in \R_{\geq 0}$ and all $i$.
\end{lem}
\begin{proof}
It is clear that $u_{2i+1}(a)$ lies in the stabilizer.  But the action of any $u_{2i}(a)$ will change the connectivity of the network, and it is impossible to return to trivial connectivity by adding more edges, or by relations.
\end{proof}

The semigroup stabilizer is too small in the sense that it does not detect the relations $u_{i\pm1}(a) u_{i}(b) \cdot L_0 = u_i(1/a + b) L_0$ used in the proof of Theorem \ref{T:efficient}.  We shall calculate the stabilizer subalgebra of the corresponding infinitesimal action of the Lie algebra $\el_{2n}$, which will in particular give an algebraic explanation of these relations.   The reason we do not work with the whole Lie group $EL_{2n}$ is threefold: (1) non-positive elements of $EL_{2n}$ will produce networks that are ``virtual'', that is, have negative edge weights; (2) the topology of $EL_{2n}$ means that to obtain an action one cannot just check the relations of Theorem \ref{thm:rels}; (3) when the parameters are non-positive, the relation Theorem \ref{thm:rels}(3) develops singularities.

To describe the infinitesimal action of $\el_{2n}$, we give derivations of $\R[x_{ij}]$, the polynomial ring in $(n+1)(n+2)/2$ variables $x_{ij}$ where $1 \leq i , j \leq n+1$ and we set 
$x_{ij} = x_{ji}$.  

\begin{prop}
The electrical Lie algebra $\el_{2n}$ acts on $\R[x_{ij}]$ via derivations as follows:
\begin{align}\label{E:2i}
e_{2i} &\mapsto \partial_{ii} + \partial_{i+1,i+1} - \partial_{i,i+1}
\\
\label{E:2i1}
e_{2i-1}& \mapsto -\sum_{1 \leq p \leq q \leq n+1} x_{ip}x_{iq} \partial_{pq}
\end{align}
\end{prop}
\begin{proof}
These formulae can be checked by directly verifying the defining relations of $\el_{2n}$.  Alternatively, they can be deduced by differentiating the formulae \eqref{eq:cim}.
\end{proof}

We calculate that
\begin{equation}\label{E:2ii}
[e_{2i},e_{2i-1}] \mapsto -x_{ii}\partial_{ii} + x_{i,i+1} \partial_{i+1,i+1}+\sum_{1 \leq p \leq n+1} \left(   x_{i+1,p} \partial_{i+1,p}-x_{ip}\partial_{ip}\right).
\end{equation}

\begin{thm}\label{thm:stab}
The stabilizer subalgebra $\el^0_{2n}$, at the zero matrix $L_0$, of the infinitesimal action of $\el_{2n}$ on the space of response matrices is generated by $e_i$ for $i$ odd, and $[e_{2i-1},e_{2i}]$ for $i=1,2,\ldots,n$.
\end{thm}
\begin{proof}
The fact the stated elements lie in $\el^0_{2n}$ follows from \eqref{E:2i1} and \eqref{E:2ii}, since $x_{ij} = 0$ at the zero matrix.  By Lemma \ref{lem:dimA}, the elements 
$e_\alpha$ in the proof of Lemma \ref{lem:dim} form a basis of $\el_{2n}$.  Write $\alpha_{ij} := \alpha_i + \alpha_{i+1}+  \cdots + \alpha_j$ for $ 1 \leq i \leq j \leq 2n$ 
to denote the positive roots of $\ssl_{2n+1}$.  Then we know that $e_{\alpha_{i,i}} \in \el^0_{2n}$ for $i$ odd, and $e_{\alpha_{i,i+1}} \in \el^0_{2n}$ for each $i$.  It follows easily that $e_{\alpha_{i,j}} \in \el^0_{2n}$ for every pair $1 \leq i \leq j \leq 2n$ where at least one of $i$ and $j$ are odd.  It follows that 
$$
\dim_\R(\el_{2n}) - \dim_\R(\el^0_{2n}) \leq  \#\{(i,j) \mid 1 \leq i \leq j \leq 2n \; \text{and $i$ and $j$ are even}\} =
n(n+1)/2.
$$
By Theorem \ref{T:efficient} the action of $(EL_{2n})_{\geq 0}$ on the zero matrix $L_0$ gives a space of dimension $n(n+1)/2$.  It follows that the above inequality is an equality, and that $\el^0_{2n}$ is generated by the stated elements.
\end{proof}

Note that a basis for $\el_{2n}/\el^0_{2n}$ is given by the $e_{\alpha_{i,j}}$'s where $i$ and $j$ are both even.  These $\alpha_{i,j}$'s are exactly the inversions of efficient permutations (cf. proof of Theorem \ref{T:efficient}).





\section{Electrical Lie algebras of finite type}
\subsection{Dimension}
Let $D$ be a Dynkin diagram of finite type and let $A = (a_{ij})$ be the associated Cartan matrix. 
To each node $i$ of $D$ associate a generator $e_i$, and define $\mathfrak e_D$ to be the Lie algebra generated over $\R$ by the $e_i$ modulo for each $i \neq j$ the relations 
$ad(e_i)^{1-a_{ij}}(e_j) = 0$ for $a_{ij} \neq -1$ and $ad(e_i)^2(e_j) = -2e_i$ for $a_{ij} = -1$.

%

These ``electrical Serre relations'' can be deduced from the type $A$ electrical Serre relations of Section \ref{sec:el} by {\it folding}.  Namely, the relation for an edge of multiplicity two ($a_{ij} = -2$) can be obtained by finding the relation for the elements $e_2$ and $e_1 + e_3$ inside $\el_3$.  Similarly, the $\e_{G_2}$ relation ($a_{ij} = - 3$) can be obtained by finding the relation for the elements $e_1+e_2+e_3$ and $e_4$ inside $\e_{D_4}$, where $e_4$ corresponds to the node of valency three in $D_4$.


Let $\mathfrak u_D$ denote the nilpotent Lie subalgebra of the simple Lie algebra of type $D$.  It is well known that $\mathfrak u_D$ is generated by Chevalley generators $\tilde e_i$ subject to the Serre relations: $ad(\tilde e_i)^{1-a_{ij}}(\tilde e_j) = 0$ for each $i \neq j$.  Thus the electrical Lie algebra $\mathfrak e_D$ is a deformation of $\mathfrak u_D$, and it follows from general results that $\dim \mathfrak e_D \leq \dim \mathfrak u_D$.  For future reference, we include a direct elementary proof of this fact.

\begin{lemma} \label{lem:dim}
 The dimension of the Lie algebra $\mathfrak e_D$ does not exceed the number of positive roots in the finite root system associated to $D$.
\end{lemma}

\begin{proof}

The Lie algebra $\mathfrak u_D$ has a basis $\tilde e_{\alpha}$ labeled by positive roots $\alpha \in R^+$ of the root system, where the $\tilde e_i$ correspond to simple roots. Let us fix an 
expression $$\tilde e_{\alpha} = [\tilde e_{i_1},[\tilde e_{i_2}, [\ldots, [\tilde e_{i_{l-1}}, \tilde e_{i_l}]\ldots]]]$$ of shortest possible length for each $\tilde e_{\alpha}$, 
and define the corresponding 
\begin{equation}\label{E:alpha}
e_{\alpha} = [e_{i_1},[e_{i_2}, [\ldots, [e_{i_{l-1}}, e_{i_l}]\ldots]]]
\end{equation}
 in $\mathfrak e_D$. We claim that the $e_{\alpha}$ span $\mathfrak e_D$. Indeed, it is enough to show that any expression $ [e_{j_1},[e_{j_2}, [\ldots, [e_{j_{\ell-1}}, e_{j_\ell}]\ldots]]]$
 lies in the linear span of the $e_{\alpha}$-s. 
Assume otherwise, and take a counterexample of smallest possible total length $\ell$.
Let us define $\hat e_\alpha$ in the free Lie algebra $\mathfrak f_D$ with generators $\hat e_i$ using \eqref{E:alpha}.
Then in $\mathfrak f_D$ we have a relation of the form
$$
 [\hat e_{j_1},[\hat e_{j_2}, [\ldots, [\hat e_{j_{\ell-1}}, \hat e_{j_\ell}]\ldots]]]
 - \sum_\alpha c_{\alpha} \hat e_{\alpha} = \hat x \in I
$$ 
where $I$ denotes the ideal generated by the Serre relations, so that $\mathfrak u_D = \mathfrak f_D/I$.  Now $\mathfrak f_D$ is naturally $\Z$-graded, and $I$ is a graded ideal, so we may assume all terms in the relation are homogeneous with the same degree.  Now replacing each instance of the Serre relation in $x$ with the corresponding electrical Serre relation gives us a relation 
$$
 [e_{j_1},[e_{j_2}, [\ldots, [e_{j_{\ell-1}}, e_{j_\ell}]\ldots]]] - \sum_\alpha c_{\alpha} e_{\alpha} = x
$$
in $\mathfrak e_D$, where $x$ is a sum of terms of the form $ [e_{k_1},[e_{k_2}, [\ldots, [e_{k_{\ell-1}}, e_{k_{\ell'}}]\ldots]]]$ where $\ell' < \ell$, and thus by assumption lies in the span of the $e_\alpha$-s.  The statement of the lemma follows.
\end{proof}

\begin{conjecture} \label{con:1}
 The dimension of $\mathfrak e_D$ coincides with the number of positive roots in the root system of $D$.
\end{conjecture}

A proof of Conjecture \ref{con:1} for Dynkin diagrams of classical type has been announced by Yi Su.  One can also define the Lie algebras $\e_D$ for any Dynkin diagram $D$, finite type or not. 

\begin{conjecture} \label{con:2}
 The Lie algebra $\e_D$ is finite dimensional if and only if $D$ is of finite type. 
\end{conjecture}

\subsection{Examples}

We illustrate Conjecture \ref{con:1} with some examples. For electrical type $A_{2n}$ it has already been verified in Theorem \ref{thm:elsp}.

\subsubsection{Electrical $B_2$}

Consider the case when $D$ has two nodes connected by a double edge. In that case we denote by $\mathfrak e_D = \mathfrak {eb}_2$ the Lie algebra generated by two generators $e$ and $f$ subject to the relations $$[e,[e,[e,f]]]=0, \;\;\; [f,[f,e]] = -2f.$$

\begin{lemma}\label{L:GL2}
The Lie algebras $\mathfrak {eb}_2$ and $\mathfrak {gl}_2$ are isomorphic.
\end{lemma}

\begin{proof}
 Consider the map $$e \mapsto \left(\begin{matrix} 1 & 1\\ 0 &  1\\ \end{matrix} \right), \;\;\; f \mapsto \left(\begin{matrix} 0 & 0\\ 1 &  0\\ \end{matrix} \right).$$ One easily checks
that it is a Lie algebra homomorphism, and that it is surjective. By Lemma \ref{lem:dim} we know the dimension of $\mathfrak {eb}_2$ is at most four, so the map must be an isomorphism.
\end{proof}

\subsubsection{Electrical $G_2$}

Consider the Lie algebra $\mathfrak {eg}_2$ generated by two generators $e$ and $f$ subject to the relations
$$[e,[e,[e,[e,f]]]]=0, \;\;\; [f,[f,e]] = -2f.$$

\begin{lemma}
The Lie algebra $\mathfrak {eg}_2$ is six-dimensional.
\end{lemma}

\begin{proof}
 According to Lemma \ref{lem:dim} the elements $e, f, [ef], [e[ef]], [e[e[ef]]], [f[e[e[ef]]]]$ form a spanning set for $\mathfrak {eg}_2$. Thus it remains to check that they are linearly independent.
This is easily done inside the following faithful representation of $\mathfrak {eg}_2$ in $\mathfrak {gl}_4$:
$$e \mapsto \left(\begin{matrix} 1 & 1 & 0 & 1\\ 0 & 1 &1 & 0 \\ 0 & 0 &1 & 0 \\0 & 0 &0 & 1 \\ \end{matrix} \right), \;\;\; 
f \mapsto \left(\begin{matrix} 0 & 0 & 0 & 0\\ 0 & 0 &0 & 0 \\ 0 & 0 &0 & 0 \\1 & 0 &0 & 0\\ \end{matrix} \right).$$
\end{proof}

\subsubsection{Electrical $C_3$}

Consider the Lie algebra $\mathfrak {ec}_3$ generated by three generators $e$, $f$ and $g$ subject to the relations
$$[e,[e,[e,f]]]=0, \;\;\; [f,[f,e]] = -2f, \;\;\; [f,[f,g]] = -2f, \;\;\; [g,[g,f]] = -2g, \;\;\; [e,g]=0.$$

\begin{lemma}
 The Lie algebra $\mathfrak {ec}_3$ is nine-dimensional. 
\end{lemma}

\begin{proof}
We consider two representations of $\mathfrak {ec}_3$: one inside $\mathfrak {sl}_9$ and one inside $\mathfrak {sl}_2$.  Let $E_{ij}$ denote a matrix with a $1$ in the $(i,j)$-th position and $0$'s elsewhere.  For the first representation, we define:
\begin{align*}
e &\mapsto E_{42} + E_{54}+E_{76}+E_{87}+E_{89}\\
f& \mapsto 2E_{24} -2 E_{26}+2E_{41}+2E_{45}-E_{47}-2E_{63}+E_{67}+E_{98} \\
g &\mapsto -E_{36} - E_{62} - E_{74} - E_{85} + E_{89}
\end{align*}


For the second representation we define
 $$e \mapsto \left(\begin{matrix} 1 & 1\\ 0 &  1\\ \end{matrix} \right), \;\;\; f \mapsto \left(\begin{matrix} 0 & 0\\ 1 &  0\\ \end{matrix} \right), 
\;\;\; g \mapsto \left(\begin{matrix} 0 & 1\\ 0 &  0\\ \end{matrix} \right).$$

One verifies directly that these are indeed representations of  $\mathfrak {ec}_3$, and the direct sum of these two representations is faithful, from which the dimension is easily
 calculated.  In fact, the first representation is simply the adjoint representation of $\ec_3$, in the basis 
$$e, f, g, [ef], [e[ef]], [fg], [e[fg]], [e[e[fg]]], [f[e[e[fg]]]],$$ and $\ec_3$ has a one-dimensional center which acts non-trivially in the second representation.
\end{proof}
%


\subsection{$Y-\Delta$ transformation of type $B$}

Let $e$ and $f$ be the generators of $\mathfrak {eb}_2$ as before, and denote by $u(t) = \exp(t e)$ and $v(t) = \exp(t f)$ the corresponding one-parameter subgroups.  The following 
proposition is a type $B$ analog of the star-triangle transformation (Proposition \ref{P:YDelta}).

\begin{prop}
 We have $$u(t_1) v(t_2) u(t_3) v(t_4) = v(p_1) u(p_2) v(p_3) u(p_4),$$ where $$p_1 = \frac{t_2 t_3^2 t_4}{\pi_2}, \;\; p_2 = \frac{\pi_2}{\pi_1}, \;\; p_3 = \frac{\pi_1^2}{\pi_2}, \;\;
p_4 = \frac{t_1 t_2 t_3}{\pi_1},$$ where $$\pi_1 = t_1 t_2 + (t_1+ t_3) t_4 + t_1 t_2 t_3 t_4, \;\;\; \pi_2 = t_1^2t_2+(t_1+t_3)^2t_4+t_1t_2t_3t_4(t_1+t_3).$$
\end{prop}

\begin{proof}
 Direct calculation inside $GL_2$, using Lemma \ref{L:GL2}. We have $$u(t) = \left(\begin{matrix} e^t & e^t t\\ 0 &  e^t\\ \end{matrix} \right), \;\;\; v(t)= \left(\begin{matrix} 1 & 0\\ t &  1\\ \end{matrix} \right),$$
and both sides of the equality are equal to $$\left(\begin{matrix} e^{t_1 + t_3} (1 + t_3 t_4 + t_1 (t_2 + t_4 + t_2 t_3 t_4)) & e^{t_1 + t_3} (t_1 + t_3 + t_1 t_2 t_3)\\ 
e^{t_1 + t_3} (t_2 + t_4 + t_2 t_3 t_4) &  e^{t_1 + t_3} (1 + t_2 t_3)\\ \end{matrix} \right).$$
\end{proof}

\begin{remark}
 One can consider a one-parameter family of deformations of the above formulas by taking 
$$\pi_1 = t_1 t_2 + (t_1+ t_3) t_4 + \tau t_1 t_2 t_3 t_4, \;\;\; \pi_2 = t_1^2t_2+(t_1+t_3)^2t_4+\tau t_1t_2t_3t_4(t_1+t_3).$$
At $\tau=0$ this specializes 
to the transformation found by Berenstein and Zelevinsky in \cite[Theorem 3.1]{BZ}. Furthermore, just like the transformation in \cite{BZ}, this deformation is given by {\it {positive}}
rational formulae.
\end{remark}

\end{document}